\documentclass[a4paper,11pt,oneside,reqno]{amsart}
\usepackage[margin=1.3in]{geometry}                % See geometry.pdf to learn the layout options. There are lots.
\usepackage{lmodern} % math, rm, ss, tt
\usepackage[OT1]{fontenc}
\usepackage{graphicx}
\usepackage{amssymb,amsmath,amsthm}
\usepackage[numbers, square]{natbib}
\usepackage{enumerate}
\usepackage{hyperref}
\usepackage{color}
\usepackage{bigints}
\allowdisplaybreaks

\usepackage{epstopdf}
\newtheorem{theorem}{Theorem}
\newtheorem{remark}[theorem]{Remark}
\newtheorem{proposition}[theorem]{Proposition}
\newtheorem{corollary}[theorem]{Corollary}
\newtheorem{lemma}[theorem]{Lemma}

\newtheorem*{thm*}{Theorem}

\begin{document}
\title{Local limit theorem in deterministic systems}
\author{Zemer Kosloff}
\thanks{The research of Z.K. was partially supported by ISF grant No. 1570/17}
\address{Einstein Institute of Mathematics,
	Hebrew University of Jerusalem, Edmond J. Safra Campus, Jerusalem 91904,
	Israel}
\email{zemer.kosloff@mail.huji.ac.il}
\author{Dalibor Volny}
\address{Laboratoire de Math\'{e}matiques Raphael Salem,
	UMR 6085, Universit\'e de Rouen, Normandie, France}
\email{dalibor.volny@univ-rouen.fr}

\maketitle

\small{\it $\ \ \ \ \ \ \ \ \ $Dedicated to Manfred Denker whose work is an inspiration for us.}
\begin{abstract}
We show that for every ergodic and aperiodic probability preserving system, there exists a $\mathbb{Z}$ valued, square integrable function $f$ such that the partial sums process of the time series $\left\{f\circ T^i\right\}_{i=0}^\infty$ satisfies the lattice local limit theorem. 
\end{abstract}

\section{introduction}
Given an ergodic, aperiodic and probability measure preserving dynamical system $\left(X,\mathcal{B},m,T\right)$, we show the existence of a measurable square integrable function $f$ with zero mean  such that its corresponding ergodic sums process  $S_n(f):=\sum_{k=0}^{n-1}f\circ T^k$ satisfies the lattice local central limit theorem. \\
A centered function $f:X\to\mathbb{R}$ satisfies the central limit theorem if for all $u\in \mathbb{R}$
\[
m\left(\frac{S_n(f)}{\left\|S_n(f)\right\|_2}\leq u\right)\xrightarrow[n\to\infty]{}\int_{-\infty}^{u}e^{-\frac{x^2}{2}}dx.
\]
A function $f:X\to\mathbb{Z}$ with $\int fdm=0$ such that $ \lim_{n\to\infty}\frac{\left\|S_n(f)\right\|_2^2}{n}=\sigma^2>0$ satisfies a lattice local central limit theorem if
\[
\sup_{x\in\mathbb{Z}}\left|\sqrt{n}\cdot m\left(S_n(f)=x\right)-\frac{e^{-x^2/\left(2n\sigma^2\right)}}{\sqrt{2\pi\sigma ^2}}\right|\xrightarrow[n\to\infty]{}0.
\]
There is also a non-lattice version of the local limit theorem which we do not consider in this paper. A function which satisfies the central limit theorem will be called a CLT function and if $f$  satisfies a local central limit theorem (whether lattice or non-lattice) we will say that $f$ is a LCLT function.\\
In case the measure theoretic entropy of the system is positive, if follows from the Sinai factor theorem that there exists a function $f:X\to \mathbb{Z}$ taking finitely many values such that the sequence $\left\{f\circ T^n\right\}_{n=0}^{\infty}$ is distributed as an i.i.d. sequence. From this it is easy to construct a LCLT function.\\
The question of existence of a central limit function for a zero entropy system such as an irrational rotation is more subtle. In the case of certain zero entropy Gaussian dynamical systems, Maruyama showed in \cite{Maruyama75} existence of CLT functions such that the variance of $S_n(f)$ grows linearly with $n$. Some years later, the seminal paper of Burton and Denker \cite{BurDen87} showed that for every aperiodic dynamical system there exists a function $f$ which satisfies the central limit theorem.\\
Denote $L_0^p = \{f\in L^2(m) : \int f dm =0\}$, $1\leq p\leq \infty$. It was observed by Burton and Denker that the set of CLT functions is dense in $L_0^2$. By \cite{Vol90} (cf. also \cite{DurVol2010},\cite{LiaVol97}) for all $1\leq p\leq \infty$, there 
exists a dense $G_\delta$ set of functions $f\in L_0^p$ for which the closure of the set of distributions of
$S_n(f)/\sqrt{n}$ contains all probability laws. A similar result (with a smaller set of limit points) holds for
$S_n(f) / ||S_n(f)||_2$. Consequently the set of CLT functions is a meager set in $L_0^p$.\\
Following \cite{BurDen87}, several extensions and improvements regarding existence and bounds of the regularity of CLT functions were done, see for example \cite{Kato87},\cite{Lacey91},\cite{Lacey93},\cite{RWeber}. The most relevant to this work is \cite{Vol98} where for every aperiodic dynamical system, a function satisfying the invariance principle and the almost sure invariance principle was constructed. More recently the question of weak convergence to other distributions was studied in \cite{AW},\cite{ThWe12}.\\
In the dynamical systems setting, it is in general a nontrivial problem to determine whether a function which satisfies the central limit theorem also satisfies the local central limit theorem. In fact, even in the nicer setting of chaotic (piecewise) smooth dynamical systems a local CLT is usually proved under more stringent spectral conditions, see for example \cite{RouEge83},\cite{GuHa88},\cite{AD01},\cite{ADSZ04} and \cite{GouLLT05}.\\
The methods of proving the central limit theorem in \cite{BurDen87} and \cite{Vol98} involved non-spectral tools which are not adapted for getting a local CLT. Moreover, the resulting partial sum process of the function takes uncountably many values. Our main theorem is the following.
\begin{thm*}[See Theorem \ref{thm: main 2}]
For every ergodic, aperiodic and probability measure preserving dynamical system $\left(X,\mathcal{B},m,T\right)$ there exists a square integrable function $f:X\to\mathbb{Z}$ with $\int fdm=0$ which satisfies the lattice local central limit theorem. 	
\end{thm*}
The construction of the aforementioned function relies on a new version of the stochastic coding theorem \cite{GrillenKre76},\cite{Kieffer80}. Namely we show that in any ergodic, aperiodic dynamical system we can realize every independent triangular array which takes finitely many values, see Proposition \ref{pro: realiz1}. We remark that for the construction of the function $f$ we need to realize a variant of a triangular array, see the beginning of Section \ref{sec: CLT} for the details.  
\subsection*{Notation}
For $z\geq 0$, the expression $x=y\pm z$ stands for $|x-y|\leq z$.  Similarly $x=a e^{\pm b}$ means $ae^{-b}\leq x\leq xe^b$. \\
For  sequences $a(n),b(n)>0$ we will write $a(n)\sim b(n)$ if $\lim_{n\to \infty}\frac{a(n)}{b(n)}=1$. In some cases it will be denoted with the little $o$ notation, that is $a(n)=b(n)+o(1)$. \\
By $a(n)=O\left(b(n)\right)$ and $a(n)\lesssim b(n)$ we mean $\limsup_{n\to\infty}\frac{a(n)}{b(n)}<\infty$.  \\
For convenience in the arithmetic arguments, $\log(\cdot)$ denotes logarithm to the base $2$ and $\ln(\cdot)$ is the standard logarithm. 
\section{The strong Alpern tower lemma and realizations of triangular arrays}
Let $A$ be a finite set, $\left\{d_n\right\}_{n=1}^\infty$ an integer valued sequence and $X_{n,m},\ n\in\mathbb{N},\ 1\leq m\leq d_n$ be an $A$-valued triangular array. In our setting this means 
\begin{itemize}
	\item For every $n\in\mathbb{N}$ and $m\in\{0,...,d_n-1\}$, $X_{n,m}:\left(\Omega,\mathcal{F},\mathbb{P}\right)\to A$.
	\item For every $n\in\mathbb{N}$, $X_{n,1},\ldots,X_{n,d_n}$ are identically distributed. 
	\item  $\left\{X_{n,m}\right\}_{n\in\mathbb{N},m\in\{1,..,d_n\}}$ is an independent array of random variables.
\end{itemize} 
This section is concerned with the following realization of triangular arrays in arbitrary ergodic measure preserving transformations. See \cite{GrillenKre76}, \cite{Kieffer80} for results in a similar flavor. 
\begin{proposition}\label{pro: realiz1}
	Let $\left(X,\mathcal{B},\mu,T\right)$ be an ergodic invertible probability measure preserving transformation. For every finite set $A$, and an $A$ valued triangular array $\mathbb{X}:=\left\{X_{n,m}\right\}_{n\in\mathbb{N},m\in\{0,..,d_n\}}$, there exists a sequence of functions $f_n:X\to A$ such that $\left\{f_n\circ T^m\right\}_{n\in\mathbb{N},m\in\{0,..,d_n-1\}}$  and $\mathbb{X} $ have the same distribution. 
\end{proposition}
The construction of the functions $f_n$ is by induction on $N\in\mathbb{N}$ so that \\$\left\{f_n\circ T^m\right\}_{1\leq n\leq N,m\in\{0,..,d_n-1\}}$  and $\left\{X_{n,m}\right\}_{1\leq n\leq N,m\in\{1,..,d_n\}}$  have the same distribution. For a fixed $N$ one can consider $\xi$, the finite partition of $X$ according to the value of the vector valued function 
\[
\left(G_N\right)_{n,m}(x):=f_n\circ T^m(x),\ \ 1\leq n\leq N,\ 0\leq m\leq d_n-1.
\]
For this reason the previous Proposition is a corollary of the following proposition. 
\begin{proposition}\label{prop: realz2}
	Let $\left(X,\mathcal{B},m,T\right)$ be an ergodic invertible probability measure preserving transformation and $\xi$ a finite measurable partition of $X$. For every finite set $A$ and $X_1,X_2,..,X_l$	a collection of $A$ valued i.i.d. random variables, there exists $f:X\to A$ such that $\left\{f\circ T^i\right\}_{i=0}^{l-1}$ and $\left\{X_i\right\}_{i=1}^{l}$ are equally distributed and $\left\{f\circ T^i\right\}_{i=0}^{l-1}$ is independent of $\xi$. 
\end{proposition}
The proof makes use of Alpern-Rokhlin castles. Given $N\in \mathbb{N}$, a $\{N,N+1\}$ Alpern-Rokhlin \textit{castle} for $\left(X,\mathcal{B},m,T\right)$ is given by two measurable sets $B_N,B_{N+1}$ such that
\begin{itemize}
	\item $\left\{T^jB_L:\ L\in\{N,N+1\},\ 0\leq j<L\right\}$ is a partition of $(X,\mathcal{B},m)$ to pairwise disjoint sets..
	\item We call $B_N,B_{N+1}$ the \textit{base elements} of the castle and the atoms in the corresponding partition are referred to as \textit{rungs}. We call $T^NB_n\uplus T^{N+1}B_{N+1}$ the \textit{top of the castle}. 
 \end{itemize}

Given $\xi$,  a finite partition of $X$, a $\{N,N+1\}$ castle is $\xi$ \textit{independent} if every rung in the castle is independent of $\xi$. All equalities of sets mean equality modulo null sets.  
\begin{proof}
	By \cite{CaMcWi18}[Corollary 1], there exists a $\xi$-independent $\{2l,2l+1\}$ castle with bases $B=B_{2l}$ and $F=B_{2l+1}$.  Note that as $T$ is invertible then,
	\[
	T\left(T^{2l-1}B\uplus T^{2l}F\right) = B\uplus F.
	\]
	%Since $T$ is recurrent then the opposite inclusion holds as otherwise there is a subset of the base of positive measure for which $T$ never returns to. We %conclude that 
	%\[
	%T\left(T^{2l-1}B\uplus T^{2l}F\right)= B\uplus F.
	%\]
	\\ Let $\zeta_1$ be the partition of the top of the tower $\{T^{2l-1}B,T^{2l}F\}$ which is its refinement according to $\vee_{i=0}^{2l} T^i\xi$. First we define $f: \cup\left\{T^{-j}C:\ C\in\ \zeta_1,\,0\leq  j< l\right\}\to A$ as follows. Partition each  $C\in \zeta_1$ to $A^l$ elements $\left\{C_{\mathbf{a}}:\  \mathbf{a}\in A^l \right\}$ such that for every $a=\left(a_i\right)_{i=0}^{l-1}\in A^l$,
	\[
	m\left(C_{\mathbf{a}}\right)=m(C)\prod_{i=0}^{l-1}\mathbb{P}\left(X_{i+1}=a_i\right).
	\] 
	and set $f(T^{-l+1+i}x)=a_i$, $i=0,\dots,l-1$, if there exists $\mathbf{a}\in A^l$ and $C\in\zeta_1$ such that $x\in C_{\mathbf{a}}$.\\
	
	It remains to define $f$ in the bottom rungs of the tower. 
	By $\zeta_1'$ we denote the refinement of $\zeta_1$ by the sets $C_{\mathbf{a}}$. By $\zeta$ we denote the joint partition of the base of the castle
	by $T^{-2l+1}(\zeta_1' \cap T^{2l-1}B)$, $T(\zeta_1' \cap T^{2l-1}B)$, $T^{-2l}(\zeta_1' \cap T^{2l}F)$, $T(\zeta_1' \cap T^{2l}F)$.
	
	For every $C\in \zeta$ we do as follows. If $C\subset B$ then we partition $C$ to $A^l$ elements $\left\{C_{\mathbf{a}}:\  \mathbf{a}\in A^l \right\}$ such that for every $\mathbf{a}=\left(a_i\right)_{i=0}^{l-1}\in A^l$,
	\[
	m\left(C_{\mathbf{a}}\right)=m(C)\prod_{i=0}^{l-1}\mathbb{P}\left(X_{i+1}=a_i\right)
	\]
	and set $f(T^{i}x)=a_i$ if there exists $\mathbf{a}\in A^l$ and $C\in\zeta$ such that $x\in C_{\mathbf{a}}$.
	
	If $C\subset \zeta\cap F$ we do the same with ${l+1}$ replacing $l$. \\
	
	For any $C\in \zeta_1$ we thus have 
	\[
	m\left( (T^{-l+1}C)\cap\ \left(\cap_{i=0}^{l-1}\left[f\circ T^i=a_{i+1}\right]\right)\right)=m(T^{-l+1}C)\prod_{i=1}^l\mathbb{P}\left(X_i=a_i\right),
	\]
	for $C\in \zeta\cap B$ we have
	\[
	m\left( C\cap\ \left(\cap_{i=0}^{l-1}\left[f\circ T^i=a_{i+1}\right]\right)\right)=m(C)\prod_{i=1}^l\mathbb{P}\left(X_i=a_i\right),
	\]
	and for $C\in \zeta\cap F$ we have
	\[
	m\left( C\cap\ \left(\cap_{i=0}^{l}\left[f\circ T^i=a_{i+1}\right]\right)\right)=m(C)\prod_{i=1}^{l+1}\mathbb{P}\left(X_i=a_i\right).
	\]
	Moreover, for $C\in \zeta_1\cap T^{2l-1}B$ 
	%with $T^{-2l+1}C \subset B$ 
	the sets $T^{-2l+1}C_{\mathbf{a}}$ are independent of $(f\circ T^i)_{i=0}^{l-1}$ 
	(conditionally on $T^{-2l+1}C$)  hence
	\begin{equation}\label{eq: eq1B}
	m\left( (T^{-2l+1}C)\cap\ \left(\cap_{i=0}^{2l-1}\left[f\circ T^i=a_{i+1}\right]\right)\right)=m(C)\prod_{i=1}^{2l}\mathbb{P}\left(X_i=a_i\right)
	\end{equation}
	and for  $C\in \zeta_1$ with $T^{-2l}C \subset F$ we have
	\begin{equation}\label{eq: eq1F}
	m\left( (T^{-2l}C)\cap\ \left(\cap_{i=0}^{2l}\left[f\circ T^i=a_{i+1}\right]\right)\right)=m(C)\prod_{i=1}^{2l+1}\mathbb{P}\left(X_i=a_i\right).
	\end{equation}
	
	Let $C\in  T^{-2l+1}(\zeta_1\cap T^{2l-1}B)$ or $C\in  T^{-2l}(\zeta_1\cap T^{2l}F)$, $0\leq k\leq l$. We have
	\[ 
	% m\left( (T^kC \cap \left(\cap_{i=0}^{2l-1-k}\left[f\circ T^i=a_{i+k+1}\right] \right) \right) =
	%m\left( C\cap  \left(\cap_{i=k}^{2l-1}\left[f\circ T^i=a_{i+k+1}\right] \right) \right) 
	T^kC \cap \left(\cap_{i=0}^{2l-1-k}\left[f\circ T^i=a_{i+k+1}\right] \right) =
	T^k\left( C\cap \left(\cap_{i=k}^{2l-1}\left[f\circ T^i=a_{i+1}\right] \right) \right);
	\] 
	by (\ref{eq: eq1B}) and (\ref{eq: eq1F}) we get
	\begin{equation}\label{eq: eq1All}
	m\left( (T^{k}C)\cap\ \left(\cap_{i=0}^{2l-1-k}\left[f\circ T^i=a_{i+k+1}\right]\right)\right)=
	m(T^kC) \prod_{i=k+1}^{2l}\mathbb{P}\left(X_i=a_{i}\right).
	\end{equation}
	
	Let us show a similar equation for $T^{-l+1}C$, $C\in \zeta_1$. 
	
	We have
	\[
	T^l \left( ( T^{-l+1}C_{\mathbf{a}}) \cap (T^{-l}B)) \cap \left(\cap_{i=l}^{2l-1}\left[f\circ T^i=a_{i+1}\right] \right)\right) = 
	(B\cap TC_{\mathbf{a}}) \cap\ \left(\cap_{i=0}^{l-1}\left[f\circ T^i=a_{i+l+1}\right]\right)
	\]    
	and the same equality holds for $F$, hence
	\begin{align*}
	&m\left(  ( T^{-l+1}C_{\mathbf{a}}) \cap (T^{-l}B)) \cap \left(\cap_{i=l}^{2l-1}\left[f\circ T^i=a_{i+1}\right] \right)\right) = \\
	&\,\,\,\,\,\,\,\,\,\,\,\,\,\,\,\,\,\,\,\,\,\,\,\,\,\,\,\,\,\,\,\,\,\,\,\,\,\,\,\,\,\,\,\,\,\,\,\,\,\,\,\,\,\,\,\,\,\,\,\,\,\,\,\,\,\,\,\,\,\,\,\, 
	m\left(  (B\cap TC_{\mathbf{a}}) \cap\ \left(\cap_{i=0}^{l-1}\left[f\circ T^i=a_{i+l+1}\right]\right)\right)
	\end{align*}    
	and
	\begin{align*}
	& m\left(  ( T^{-l+1}C_{\mathbf{a}}) \cap (T^{-l}F)) \cap \left(\cap_{i=l}^{2l-1}\left[f\circ T^i=a_{i+1}\right] \right)\right) = \\
	&\,\,\,\,\,\,\,\,\,\,\,\,\,\,\,\,\,\,\,\,\,\,\,\,\,\,\,\,\,\,\,\,\,\,\,\,\,\,\,\,\,\,\,\,\,\,\,\,\,\,\,\,\,\,\,\,\,\,\,\,\,\,\,\,\,\,\,\,\,\,\,\, 
	m\left(  (F\cap TC_{\mathbf{a}}) \cap\ \left(\cap_{i=0}^{l-1}\left[f\circ T^i=a_{i+l+1}\right]\right)\right).
	\end{align*}  
	Therefore, 
	\[ m\left(  ( T^{-l+1}C_{\mathbf{a}})  \cap \left(\cap_{i=l}^{2l-1}\left[f\circ T^i=a_{i+1}\right] \right)\right) =
	m\left(  ( TC_{\mathbf{a}}) \cap\ \left(\cap_{i=0}^{l-1}\left[f\circ T^i=a_{i+l+1}\right]\right)\right)
	\]    
	and by independence of $\zeta$ and $f,\dots,f\circ T^{l-1}$ on $B\cup F$ we deduce 
	\begin{align*} 
	m\left(  ( T^{-l+1}C_{\mathbf{a}})  \cap \left(\cap_{i=l}^{2l-1}\left[f\circ T^i=a_{i+1}\right] \right)\right) &=
	m( TC_{\mathbf{a}})m\left( (B\cup F)\cap\left(\cap_{i=0}^{l-1}\left[f\circ T^i=a_{i+l+1}\right]\right) \right)\\ 
	&=m(C_{\mathbf{a}})\prod_{i=1}^{l}\mathbb{P}\left(X_i=a_{i+l+1}\right)\\
	&=m\left(C\right)\prod_{i=1}^{2l}\mathbb{P}\left(X_i=a_{i}\right).
	\end{align*}
	Recall that if $C\in \zeta_1$ and $\mathbf{a} =\left(a_i\right)_{i=1}^{l}\in A^l$ then
	\[
	T^{-l+1}C_{\mathbf{a}} = T^{-l+1}C \cap\ \left(\cap_{i=0}^{l-1}\left[f\circ T^i=a_{i+1}\right]\right).
	\]
	Therefore,
	\begin{align*}
	m\left(  (T^{-l+1}C)\cap\ \left(\cap_{i=0}^{2l-1}\left[f\circ T^i=a_{i+1}\right]\right)\right)&= 
	m\left(  ( T^{-l+1}C_{\mathbf{a}})  \cap \left(\cap_{i=l}^{2l-1}\left[f\circ T^i=a_{i+1}\right] \right)\right)\\
	&=m\left(C\right)\prod_{i=1}^{2l}\mathbb{P}\left(X_i=a_{i}\right),
	\end{align*}
	hence
	\[
	m\left(  (T^{-l+1}C)\cap\ \left(\cap_{i=0}^{2l-1}\left[f\circ T^i=a_{i+1}\right]\right)\right)=
	m((T^{-l+1}C))  \prod_{i=1}^{2l}\mathbb{P}\left(X_i=a_{i}\right).
	\]
	For $0\leq k\leq l-1$ we thus have
	\begin{equation}\label{eq: eq2}
	m\left(  (T^{-l+k+1}C)\cap\ \left(\cap_{i=0}^{2l-1}\left[f\circ T^i=a_{i+1}\right]\right)\right)=
	m((T^{-l+1}C))  \prod_{i=1}^{2l}\mathbb{P}\left(X_i=a_{i}\right).
	\end{equation}
	\ 
	We show that $f$ satisfies the conclusion of the proposition. Let $D\in \xi$ and $\mathbf{a}\in A^l$. We claim that
	\begin{equation}\label{eq: main in prop2}
	m\left(D\cap\ \left(\cap_{i=0}^{l-1}\left[f\circ T^i=a_{i+1}\right]\right)\right)=m(D)\prod_{i=1}^l\mathbb{P}\left(X_i=a_i\right).
	\end{equation}
	Let $Y$ be a rung of the castle. Since the castle is $\xi$-independent, $m(D\cap Y)=m(D)m(Y)$ and $D\cap Y$ is a finite union of sets of the form 
	$T^{-r}C$ with $C\in\zeta_1$ and $r$ a fixed integer, $0\leq r\leq 2l-1, 2l$.
	(depending on the tower of the castle to which $Y$ belongs). Denote $C'= T^{-r}C$.
	
	From (\ref{eq: eq1All}) and (\ref{eq: eq2}) it follows that 
	\begin{equation}\label{eq: eq3}
	m\left(C'\cap \ \left(\cap_{i=0}^{l-1}\left[f\circ T^i=a_{i+1}\right]\right)\right)=m(C)\prod_{i=1}^l\mathbb{P}\left(X_i=a_i\right).
	\end{equation}
	Because $\zeta_1$ is finer than $(T^{2l-1}B \cup T^{2l}F) \cap \vee_{i=0}^{2l} T^i\xi$, $D\cap Y$ is a union of sets $C'$. $D$ is just a disjoint union 
	of the sets $C'$ over all rungs $Y$. Summing equations (\ref{eq: eq3}) for all such $C'$ we see that equation (\ref{eq: main in prop2}) holds.

\end{proof}
\section{Definition of the function and proof of the CLT}\label{sec: CLT}
Let $\left(X,\mathcal{B},\mu,T\right)$ be an ergodic invertible probability measure preserving transformation and $U:L_2(X,\mu)\to L_2(X,\mu)$ is its corresponding Koopman operator. In this susbsection we construct the function $f$  for which the local limit theorem holds. \\
For $k\in\mathbb{N}$ we define 
\[
p_k:=\begin{cases}
2^k, & k\ \text{even} \\
2^{k-1}+1, & k\ \text{odd}.
\end{cases}
\]
and
\[
d_k:=2^{k^2},\ \ \text{and}\ \ \alpha_k:=\frac{1}{p_k\sqrt{k\log k}}.
\]
By a repeated inductive iteration of Proposition \ref{prop: realz2}, there exists a sequence of functions $\bar{f}_k:X\to \left\{-1,0,1\right\}$ such that:
\begin{itemize}
	\item[(a)] For every $k\in\mathbb{N}$, $\left\{\bar{f}_k\circ T^j\right\}_{j=0}^{2d_k+p_k}$ is an i.i.d. sequence, $\{-1,0,1\}$ valued and 
	\[
	\mu\left(\bar{f_k}=1\right)=\mu\left(\bar{f_k}=-1\right)=\frac{\alpha_k^2}{2}.
\]
\item[(b)] For every $k\geq 2$, the finite sequence $\left\{\bar{f}_k\circ T^j\right\}_{j=0}^{2d_k+p_k}$ is independent of\\ $\mathcal{A}_k:=\left\{\bar{f}_l\circ T^j:\ 1\leq l<k,\ 0\leq j\leq 2d_k+p_k \right\}$. 
\end{itemize}  
To explain how we apply the proposition, the values of the functions $g\in\mathcal{A}_k$ induce a partition $\xi$ of $X$ into $3^{2d_k+p_k}$ elements. Write $X_1,X_2,X_3,...,X_{2d_k+p_k}$ for an i.i.d. sequence such that $X_1$ is $\{-1,0,1\}$ distributed with $\mathbb{E}\left(X_1\right)=0$ and $\mathbb{E}\left(X_1^2\right)=\alpha_k^2$ and apply Proposition \ref{prop: realz2}. In this definition the sequence of functions 
\[
\left\{\bar{f}_k\circ T^j: k\in\mathbb{N},\ 0\leq j\leq 2d_{k}+p_k \right\}
\]
is a triangular array. However, Proposition \ref{prop: realz2} enables us to get property (b) which is more than a realization of this triangular array. This step is crucial in what follows.\\
Let us define $f:=\sum_{k=1}^\infty f_k$ where for each $k\in\mathbb{N}$,
\[
f_k:=\sum_{i=0}^{p_k-1}U^i\bar{f}_k-U^{d_k}\left(\sum_{i=0}^{p_k-1}U^i\bar{f}_k\right). 
\]
It is worth to note that each of the function $f_k$ is a coboundary with transfer function
\[
g_k:=\sum_{j=0}^{d_k-1}\sum_{i=0}^{p_k-1}U^{i+j}\bar{f}_k,\ \ \ f_k=g_k-Ug_k.
\]
\begin{proposition}
$f\in L^2(X,\mu)$.
\end{proposition}
\begin{proof}
Fix $k\in\mathbb{N}$ and write $V_{k,i}:=U^i\bar{f}_k-U^{d_k+i}\bar{f}_k$ where $i\in\{0,1,...,p_k-1\}$. This is a sequence of i.i.d. random variables with $\int_XV_{k,i}d\mu=0$ and $\left\|V_{k,i}\right\|_2^2=2\alpha_k^2$. Therefore 
\begin{align*}
\left\|f_k\right\|_2^2&= \int_X \left(\sum_{i=0}^{p_k-1}V_{k,i}\right)^2dm\\
&=\sum_{i=0}^{p_k-1}\int_X \left(V_{k,i}\right)^2dm=2\alpha_k^2p_k\leq \frac{2}{p_k}.
\end{align*}
By condition (b) in the definition of the $\bar{f}_k$'s, the functions $f_1,f_2,..$ are independent. As they are also centered,
\[
\left\|f\right\|_2^2=\sum_{k=1}^\infty \left\|f_k\right\|_2^2\leq \sum_{k=1}^\infty \frac{2}{p_k}<\infty.
\]
We conclude that $f$ is well defined. 
\end{proof}
\begin{theorem}\label{thm: main 2}
The function $f:X\to\mathbb{R}$ satisfies the local limit theorem with $\sigma^2:=2(\ln 2)^2$. That is 
\[
\sup_{x\in\mathbb{Z}}\left|\sqrt{n}\mu\left(S_n(f)=x\right)-\frac{e^{-x^2/\left(2n\sigma^2\right)}}{\sqrt{2\pi\sigma ^2}}\right|\xrightarrow[n\to\infty]{}0
\]
\end{theorem}
\subsubsection*{Discussion on the steps in the proof of Theorem \ref{thm: main 2}}

The beginning of the proof is done by an argument similar to the one in Terrence Tao's blogpost on local limit theorems. The first step, which is done in the next subsection, is to prove the central limit theorem. This is done by calculating the second moments and verifying the Lindeberg condition. The choice of $d_k$ instead of an exponential sequence as in \cite{Vol98} is used in this step.\\
In the course of the proof of the CLT, $S_n(f)$ is decomposed into a sum of several independent random variables and we identify the main term, which we will call in this discussion $Y_n$. By Proposition \ref{prop: essential part}, the local limit theorem for $S_n(f)$ is equivalent to the local limit theorem for the main term $Y_n$. \\
In Section \ref{sec: LLT}, we show the local limit theorem for $Y_n$. There we use Fourier inversion and the CLT to reduce the local limit theorem to a question about uniform integrability of certain functions. In this step the choice of $p_k$ will help with a strong aperiodicity type statement which appears in the proof of Lemma \ref{lem: dom away from 0}. One problem we encounter, which is not present in the proof of classical local limit theorems, is that it seems a difficult problem to control the Fourier expansion of $Y_n$ around zero for an interval of fixed size (or a scaled interval of length constant times $\sqrt{n}$). We overcome this problem by obtaining a sharp enough aperiodicity bound which reduces the estimate around zero to an interval of length constant times $\sqrt[4]{n}$ as in Lemma \ref{lem: dom near 0}.  

\subsection{Proof of the CLT}

We start by presenting $S_n(f)$ as a sum of three terms depending on the scale of $k$ with respect to $n$. That is 
\[
S_n(f)=Z_{Sm}(n)+\hat{Y}(n)+Z_{La}(n).
\]
where
\begin{align*}
Z_{Sm}(n)&:= \sum_{k:\ d_k\leq n}S_n\left(f_k\right)\\
\hat{Y}(n)&:=\sum_{k:\ p_k< n<d_k}S_n\left(f_k\right)\\
Z_{La}(n)&:=\sum_{k:\ n\leq p_k}S_n\left(f_k\right)
\end{align*}
\begin{lemma}\label{lem: 1st for CLT}
For every $n\in\mathbb{N}$ the random variables $Z_{Sm}(n),\hat{Y}(n),Z_{La}(n)$ are independent and
\begin{itemize}\label{lem: 1st Lem in CLT}
	\item[(a)] $\left\|Z_{Sm}(n)+Z_{La}(n)\right\|_2^2=O\left(\frac{n}{\sqrt{\log n}}\right)$.
	\item[(b)]  $\frac{1}{\sqrt{n}}\hat{Y}(n)-\frac{1}{\sqrt{n}}S_n(f)\xrightarrow[n\to\infty]{}0$ in $L^2(X,\mu)$.
\end{itemize}
\end{lemma}
\begin{proof}
For each $k\in \mathbb{N}$, $S_n\left(f_k\right)$ is a sum of functions from the sequence $\left\{\bar{f}_k\circ T^i\right\}_{i=0}^{n+d_k+p_k-1}$. For all $k$'s appearing in the sums describing $\hat{Y}(n)$ and $Z_{La}(n)$, one has $n< d_k$, therefore $\hat{Y}(n)$ and $Z_{La}(n)$ are sums of functions of the form $\left\{\bar{f}_k\circ T^i\right\}_{i=0}^{2d_k+p_k}$ with $k> \sqrt{\log n}$ and $Z_{Sm}(n)$ is a sum of functions from $\left\{\bar{f}_k\circ T^i\right\}_{i=0}^{2d_{k^*}+p_{k^*}}$ with $k\leq\sqrt{\log n}$ and $k^*$ is the smallest integer such that $k^*>\sqrt{\log n}$.\\
By property (b) in the definition of the  $\bar{f}_k$'s we see that $Z_{Sm}(n)$ is independent of $\hat{Y}(n)$ and $Z_{La}(n)$. A similar reasonning using property (b) of the functions $\bar{f}_k$ shows that $Z_{Sm}(n),\hat{Y}(n), Z_{La}(n)$ are independent. \\
We now turn to prove part (a). Since $f_k=g_k-Ug_k$ then 
\[
Z_{Sm}(n)=\sum_{k:\ d_k\leq n}\left(g_k-U^ng_k\right).
\]
By independence of $\{\bar{f}_k\circ T^i\}_{i=0}^{p_k+d_k}$, 
\begin{align*}
\left\|g_k\right\|_2^2 &=\sum_{j=0}^{d_k+p_k-1} \left(\#\left\{(l,i)\in \left[0,d_k\right)\times \left[0,p_k\right):\ l+i=j \right\}\right)^2 \left\|\bar{f}_k\right\|_2^2\\
&< \alpha_k^2p_k^2\left(p_k+d_k\right)\leq \frac{2^{k^2+1}}{k\log k}. 
\end{align*}
The functions $\{g_k\}_{k=1}^\infty$ are centered and independent, thus 
\begin{align*}
\left\|\sum_{k:\ d_k\leq n}g_k\right\|_2^2&=\sum_{k:\ d_k\leq n}\left\|g_k\right\|_2^2\\
&\leq 2\sum_{k:d_k\leq n}\frac{2^{k^2}}{k\log k}\\
&=2\sum_{k:k\leq \sqrt{\log n}}\frac{2^{k^2}}{k\log k}=O\left(\frac{n}{\sqrt{\log n}}\right), 
\end{align*}
where the last inequality follows from Lemma \ref{lem: app1}. As $U$ is unitary it follows that 
\begin{equation}\label{eq: Zsm}
\left\|Z_{Sm}(n)\right\|_2^2=O\left(\frac{n}{\sqrt{\log n}}\right). 
\end{equation}
It remains to bound $\left\|Z_{La}(n)\right\|_2^2$. First note that for $k$ such that $n\leq p_k$, by independence of the summands
\begin{align*}
\left\|S_n\left(f_k\right)\right\|_2^2&=2\sum_{j=0}^{p_k-1+n}\left(\left\|\bar{f}_k\right\|_2\#\left\{(i,l)\in \left[0,p_k-1\right]\times [0,n-1]:i+l=j\right\}\right)^2 \\
&\leq 2n^2\left(n+p_k\right)\alpha_k^2\leq 4n^2\left(\frac{1}{p_k k\log k}\right)
\end{align*}
where the last inequality holds as $n\leq p_k$. The random variables $\left\{S_n\left(f_k\right)\right\}_{\left\{k:\ p_k\geq n\right\}}$ are independent and their sum is $Z_{La}(n)$, therefore 
\[
\left\|Z_{La}(n)\right\|_2^2=\sum_{k:\ n\leq p_k}\left\|S_n\left(f_k\right)\right\|_2^2\leq 4n^2\sum_{k:\ n\leq p_k} \frac{1}{p_kk\log k}. 
\]
Since in addition, $p_k\leq 2^k$, then
\[
\sum_{k:\ n\leq p_k} \frac{1}{p_kk\log k}\leq \frac{1}{\log n}\sum_{k:\ n\leq p_k} \frac{1}{p_k}\leq \frac{4}{n\log n}.
\]
and $\left\|Z_{La}(n)\right\|_2^2=O\left(\frac{n}{\log n}\right)$. This together with \eqref{eq: Zsm} implies part (a). Part (b) is a direct consequence of part (a). 
\end{proof}
\begin{lemma}\label{lem: 2nd for CLT}
$\lim_{n\to\infty}\left(\frac{1}{n}\left\|S_n(f)\right\|_2^2\right)=2(\ln2)^2=:\sigma ^2$
\end{lemma}
\begin{proof}
By Lemma \ref{lem: 1st Lem in CLT}.(b) it is enough to show that 
\begin{equation}\label{eq: Yes}
\lim_{n\to\infty}\frac{\left\|\hat{Y}(n)\right\|_2^2}{n}=2(\ln 2)^2.
\end{equation} For a $k$ with $p_k<n<d_k$ we have,
\[
S_n\left(f_k\right)=\sum_{j=0}^{n-1}\sum_{i=0}^{p_k-1}U^{i+j}\bar{f}_k-U^{d_k}\sum_{j=0}^{n-1}\sum_{i=0}^{p_k-1}U^{i+j}\bar{f}_k
\]
hence 
\[
\hat{Y}(n)=\sum_{k:p_k<n<d_k}S_n\left(f_k\right)=\sum_{k:p_k<n<d_k}\left[\left(A_k+B_k+C_k\right)-U^{d_k}\left(A_k+B_k+C_k\right)\right]
\]
where
\[
A_k:=\sum_{i=0}^{p_k-2}(i+1)U^i\bar{f}_k,\ \ B_k:=p_k\sum_{i=p_k-1}^{n-1}U^i\bar{f}_k,\ \ C_k:=\sum_{i=n}^{n+p_k-2}\left(n+p_k-1-i\right)U^i\bar{f}_k.
\]
Using independence of $U^i\bar{f}_k,\ 0\leq i\leq p_k-1$, we deduce 
\[
\left\|A_k\right\|_2^2\leq \sum_{i=0}^{p_k-2}(i+1)^2\frac{1}{p_k^2k\log k}\leq \frac{p_k}{k\log k}\lesssim \frac{2^{k}}{k}
\]
By Lemma \ref{lem: app2} we derive,
\begin{align*}
\sum_{k:p_k<n<d_k}\left\|A_k\right\|_2^2&\leq  \sum_{k:p_k<n<d_k} \frac{2^{k}}{k}\\
&\lesssim \sum_{k:\ \sqrt{\log n}\leq k\leq \log n} \frac{2^{k}}{k}=O\left(\frac{n}{\log n}\right).
\end{align*}
Similarly we derive 
\[
\sum_{k:p_k<n<d_k}\left\|C_k\right\|_2^2= O\left(\frac{n}{\log n}\right). 
\]
Finally, 
\begin{align*}
\sum_{k:p_k<n<d_k}\left\|B_k\right\|_2^2&=\sum_{k:p_k<n<d_k} \left(n+1-p_k\right)\alpha_k^2p_k^2\\
&\sim \left[n\sum_{k:p_k<n<d_k} \frac{1}{k\log k}-\sum_{k:p_k<n<d_k} \frac{p_k}{k\log k}\right].
\end{align*}
By Lemma \ref{lem: app2}, 
\begin{align*}
\sum_{k:p_k<n<d_k} \frac{p_k}{k\log k}&\lesssim  \sum_{k:\sqrt{\log n}<k<\log n} \frac{2^k}{k}=O\left(\frac{n}{\log n}\right).
\end{align*}
Since $\lim_{k\to \infty} \frac{\log p_k}{k}=1$, then  \footnote{$\bigint_{\sqrt{\log n}}^{\log n}\frac{dx}{x\log x}=\ln2 \left(\ln\ln (\log x)-\ln\ln (\sqrt{\log x})\right)=(\ln2)^2$}
\begin{align*}
\sum_{k: p_k<n<d_k} \frac{1}{k\log k}&\sim \sum_{k:\sqrt{\log n}<k<\log n} \frac{1}{k\log k}\\
&\sim \bigints_{\ \sqrt{\log n}}^{\log n}\frac{dx}{x\log x}=\left(\ln 2\right)^2. 
\end{align*}
The random variables $A_k,B_k,C_k,U^{d_k}A_k,U^{d_k}B_k, U^{d_k}C_k$, where the index $k$ satisfies $p_k<n<d_k$  are all independent, whence 
\[
\frac{\left\|\hat{Y}(n)\right\|_2^2}{n}=\frac{2}{n}\sum_{k:p_k<n<d_k}\left(\left\|A_k\right\|_2^2+\left\|B_k\right\|_2^2+\left\|C_k\right\|_2^2\right)=2(\ln 2)^2+o(1),			
\] 
showing that \eqref{eq: Yes} holds. 
\end{proof}
Define for $i\in\{1,..,n\}$ a function $Y_i(n):X\to \mathbb{Z}$ by
\[
Y_i(n):=\sum_{\left\{k:p_k\leq i+1,\ p_k<n<d_k\right\}}p_k\left(U^i\bar{f}_k-U^{d_k+i}\bar{f}_k\right).
\]
Because
\[
\sum_{i=1}^{n-1}Y_i(n)=\sum_{\left\{k:p_k<n<d_k\right\}}\left(B_k-U^{d_k}B_k\right),
\]
the following is deduced from  Lemma \ref{lem: 1st Lem in CLT} and the proof of Lemma \ref{lem: 2nd for CLT}. 
\begin{proposition}\label{prop: decomp for CLT}
For every $n\in\mathbb{N}$, the random variables $S_n(f)-\sum_{i=1}^nY_i(n)$ and $\sum_{i=1}^nY_i(n)$ are independent and 
\[
\left\|S_n(f)-\sum_{i=1}^{n-1} Y_i(n)\right\|_2^2=O\left(\frac{n}{\sqrt{\log n}}\right).
\]
\end{proposition}
\begin{proposition}
$S_n(f)$ converges in distribution to the normal law $\mathcal{N}(0,\sigma ^2)$ with $\sigma^2=2\left(\ln 2\right)^2$.	
\end{proposition}
\begin{proof}
By Proposition \ref{prop: decomp for CLT} it is sufficient to prove the convergence for $\frac{1}{\sqrt n}\sum_{i=1}^n Y_i(n)$. Since for every $n\in\mathbb{N}$, the random variables $Y_1(n),...,Y_n(n)$ are independent and centered, this will follow once we verify the Lindeberg's condition
\[
\forall \epsilon>0,\ \frac{1}{n\sigma ^2}\sum_{i=1}^{n-1} \int_X \left(Y_i(n)\mathbf{1}_{\left[Y_i(n)^2>\epsilon^2\sigma^2 n\right]}\right)^2d\mu\xrightarrow[n\to\infty]{}0.
\] 
 Because $p_k\leq 2^k$, if $J\subset\mathbb{N}$ is such that for all $j\in {J}$, $2^j<\epsilon\sigma \frac{\sqrt{n}}{8}$ then 
 \[
 \left|\sum_{j\in J} p_j\left(U^i\bar{f}_j-U^{d_j+i}\bar{f}_j\right)\right|\leq 2\sum_{j\in {J}}2^j\leq\epsilon\sigma\frac{\sqrt{n}}{2}.
 \] 
 Therefore, writing $A_{n,\epsilon}:=\left\{k: \log (\sigma\epsilon)-1+\frac{1}{2}\log n\leq k\leq \log n\right\}$, if $\left|Y_i(n)\right|>\epsilon\sigma \sqrt{n}$ then,
 \[
\left|Y_i(n)\right| \leq 2\left|\sum_{k\in A_{n,\epsilon}}p_k\left(U^i\bar{f}_k-U^{d_k+i}\bar{f}_k\right)\right|.
 \]
 We conclude that 
 \[
 \left(Y_i(n)\mathbf{1}_{\left[Y_i(n)^2>\epsilon^2\sigma^2n\right]}\right)^2\leq 4\left(\sum_{k\in A_{n,\epsilon}}p_k\left(U^i\bar{f}_k-U^{d_k+i}\bar{f}_k\right)\right)^2.
 \]	
The terms $p_kU^i\bar{f}_k$ which appear in the right hand side are mutually independent, thus for all $i\in\{1,..,n-1\}$,
\begin{align*}
\int_X \left(Y_i(n)\mathbf{1}_{\left[Y_i(n)^2>\epsilon^2\sigma^2n\right]}\right)^2d\mu&\leq 4\int_X \left(\sum_{k\in A_{n,\epsilon}}p_k\left(U^i\bar{f}_k-U^{d_k+i}\bar{f}_k\right)\right)^2d\mu\\
&=8\sum_{k\in A_{n,\epsilon}}p_k^2\left\|\bar{f}_k\right\|_2^2\\
&= 8\sum_{k\in A_{n,\epsilon}}\frac{1}{k\log k}\\
&\lesssim \ln\ln \log(n)-\ln\ln\left(\log(\sigma\epsilon)-1+\frac{1}{2}\log n\right)=o(1),
\end{align*} 
as $n\to\infty$. This proves that the Lindeberg condition holds. 
\end{proof}

\section{Proof of the local CLT}\label{sec: LLT}
For the proof of the local CLT we first start with a new presentation of the main term $\sum_{i=1}^n Y_i(n)$. For this purpose let
\[
I_n:=\left\{k\in 2\mathbb{N}:\ 2^k<n<2^{k^2} \right\}
\]
and for $k\in I_n$ set 
\[
V_k:=\sum_{i=2^k}^{n-1}\left(p_k\left(U^i\bar{f}_k-U^{i+d_k}\bar{f}_k\right)+p_{k+1}\left(U^i\bar{f}_{k+1}-U^{i+d_{k+1}}\bar{f}_{k+1}\right)\right)
\]
To explain, $I_n$ roughly denotes the even integers in the segment $p_k<n<d_k$ and for each $k\in I_n$,   $V_k$ is, up to an independent term with small second moment, almost equal to $B_{k}+B_{k+1}-U^{d_k}\left(B_k\right)-U^{d_{k+1}}\left(B_{k+1}\right)$. The next Proposition makes this claim precise. We use the notation
\[
\mathsf{U}_n:=\sum_{k\in I_n}V_k,\ \ W_n:=\sum_{i=1}^nY_i(n)
\]
\begin{proposition}
The random variables $\mathsf{U}_n$ and $E_n:=W_n-\mathsf{U}_n$ are independent and 
\[
\left\|E_n\right\|_2^2=O\left(\frac{n}{\sqrt{\log n}}\right).
\]
\end{proposition} 
\begin{proof}
There are three types of terms appearing in $E_n$. The first is if there exists an even integer $k$ such that $p_k<n<d_k$ and $n\leq p_{k+1}=2^k+1$. Since for $k$ even, $p_k=2^k$, this happens if and only if $\log(n-1)=k\in 2\mathbb{N}$. In this case, writing $k=\log(n-1)$, then $V_{\log(n-1)}$ contains the term
\begin{align*}
\sum_{i=2^k}^{n-1} p_{k+1}\left(U^i\bar{f}_{k+1}-U^{i+d_{k+1}}\bar{f}_{k+1}\right)&= n\left(U^{n-1}\bar{f}_{\log(n-1)+1}-U^{n-1+d_{\log(n-1)+1}}\bar{f}_{\log(n-1)+1}\right)\\
&=A(n),
\end{align*}
where we have used that $2^k=n-1$ and $p_{k+1}=n$.\\
The second type is when there exists an even $k$ such that $d_k\leq n<d_{k+1}$. In this case $p_{k+1}<n<d_{k+1}$, therefore  $B_{k+1}-U^{d_{k+1}}B_{k+1}$ appears in $W_n$ and does not appear in $\mathsf{U}_n$. Since $n<d_{k+1}$ we see that 
\[
\sqrt{\log n}<k+1.
\]
This implies that 
\begin{equation}
\left\|B_{k+1}\right\|_2^2=\left(n+1-p_{k+1}\right)p_{k+1}^2\alpha_{k+1}^2\leq \frac{n}{k+1}\leq \frac{n}{\log n}.
\end{equation}
Finally the third type of terms comes from the fact that for each $k\in I_n$ we are not including $p_k\left(U^{p_k-1}\bar{f}_k-U^{d_k+p_k-1}\bar{f}_k\right)$ in the definition of $V_k$ while it does appear in $B_k$. \\
We conclude that
\begin{align*}
E_n&=-\mathbf{1}_{\left[\log(n-1)\in2\mathbb{N}\right]}A(n)+\sum_{k\in I_n}p_k\left(U^{p_k-1}\bar{f}_k-U^{d_k+p_k-1}\bar{f}_k\right)\\
&\ \ \  +1_{\exists k\in 2\mathbb{N}: \left[d_k\leq n<d_{k+1}\right]}\left(B_{k+1}-U^{d_{k+1}}B_{k+1}\right).
\end{align*}
The independence of $E_n$ and $\mathsf{U}_n$ follows from properties (a) and (b) in the construction of the functions $\bar{f}_k$. Finally as the terms in the sum of $E_n$ are independent, using the bound on the last term (if and when it appears)
\begin{align*}
\left\|E_n\right\|_2^2&=2\left(n^2\mathbf{1}_{\left[\log(n-1)\in2\mathbb{N}\right]}\left\|\bar{f}_{\log(n-1)+1}\right\|_2^2+ \sum_{k\in I_n}p_k^2\left\|\bar{f}_k\right\|_2^2\right)+O\left(\frac{n}{\sqrt{\log n}}\right)\\
&\leq 2\left(n^2\left(\alpha_{\log(n-1)}\right)^2+\sum_{k\in I_n}p_k^2\alpha_{k}^2 \right)+O\left(\frac{n}{\sqrt{\log n}}\right)\\
&= o(1)+\sum_{k:\ \sqrt{\log(n)}<k<\log(n)}\frac{1}{k\log k}+O\left(\frac{n}{\sqrt{\log n}}\right)=O\left(\frac{n}{\sqrt{\log n}}\right).
\end{align*}
\end{proof}
Combining this with Proposition \ref{prop: decomp for CLT}, we have shown. 
\begin{corollary}\label{cor: punch line decomposition}
The random variables $S_n(f)-\mathsf{U}_n$ and $\mathsf{U}_n$ are independent and 
\[
\left\|S_n(f)-\mathsf{U}_n\right\|_2^2=O\left(\frac{n}{\sqrt{\log(n)}}\right).
\]
Consequently,  $\frac{1}{\sqrt{n}}\mathsf{U}_n$ converges in distribution to a normal law with variance $\sigma^2=2(\ln n)^2$.  
\end{corollary}
In the remaining part of this section we will prove that $\mathsf{U}_n$ satisfies a local CLT and use Proposition \ref{prop: essential part} to deduce the local CLT for $S_n(f)$. 
\begin{theorem}\label{thm: main LCLT part}
Writing $\sigma^2=2(\ln n)^2$ then,
\[
\sup_{x\in\mathbb{Z}}\left\|\sqrt{n}\mu \left(\mathsf{U}_n=x\right)-\frac{1}{\sqrt{2\pi \sigma^2}}e^{-x^2/2n\sigma^2}\right|\xrightarrow[n\to\infty]{}0.
\]
\end{theorem}
\begin{proof}[Deduction of Theorem \ref{thm: main 2}]
By Corollary \ref{cor: punch line decomposition}, the random variables $S_n(f)$ and $\mathsf{U}_n$ satisfy the conditions of $X_n$ and $Y_n$ in Proposition \ref{prop: essential part}. Thus by Theorem \ref{thm: main LCLT part} and Proposition \ref{prop: essential part} we see that 
\[
\sup_{x\in\mathbb{Z}}\left\|\sqrt{n}\mu \left(S_n(f)=x\right)-\frac{1}{\sqrt{2\pi \sigma^2}}e^{-x^2/2n\sigma^2}\right|\xrightarrow[n\to\infty]{}0.
\] 
\end{proof}
For the proof of Theorem \ref{thm: main LCLT part} introduce the characteristic function of $\mathsf{U}_n$,
\[
\phi_n(t):=\int \exp\left(it\mathsf{U}_n\right)d\mu.
\]
The following two Lemmas are the core estimates which are used in the domination part, as in the proof of the local CLT in Terrence Tao's blog.
\begin{lemma}\label{lem: dom away from 0}
There exists $c>0$ such that for all $\sqrt[4]{n}\leq |x|\leq \pi \sqrt{n}$,
\[
\left|\phi_n\left(\frac{x}{\sqrt{n}}\right)\right|\leq \exp\left(-c\sqrt[4]{n}\right)\leq \exp\left(-d\sqrt{|x|}\right)
\]
where $d:=\frac{c}{\sqrt{\pi}}$.
\end{lemma}  
\begin{lemma}\label{lem: dom near 0}
There exists $N\in\mathbb{N}$ and a constant $L>0$ such that for all $n>N$ and $|x|\leq \sqrt[4]{n}$,
\[
\left|\phi_n\left(\frac{x}{\sqrt{n}}\right)\right|\leq \exp\left(-Lx^2\right).
\]	
\end{lemma}
\begin{proof}[Proof of Theorem \ref{thm: main LCLT part}]
Let $m\in\mathbb{Z}$. By Fourier inversion,
\[
\mu\left(\mathsf{U}_n=m\right)=\frac{1}{2\pi}\int_{-\pi}^{\pi}\phi_n(t)e^{-itm}dt
\]
Applying the change of variable, $t=\frac{x}{\sqrt{n}}$, we see that
\[
\sqrt{n}\mu\left(\mathsf{U}_n=m\right)=\frac{1}{2\pi}\int_{-\pi\sqrt{n}}^{\pi\sqrt{n}}\phi_n\left(\frac{x}{\sqrt{n}}\right)e^{-ixm/\sqrt{n}}dx
\]
Since, 
\[
\frac{1}{\sqrt{2\pi}\sigma}e^{-m^2/2n\sigma^2}=\frac{1}{2\pi}\int_{\mathbb{R}} e^{-\sigma^2x^2/2}e^{-ixm/\sqrt{n}}dx
\]
then it remains to show that
\[
\sup_{m\in\mathbb{Z}}\left|\frac{1}{2\pi}\int_{-\pi\sqrt{n}}^{\pi\sqrt{n}}\phi_n\left(\frac{x}{\sqrt{n}}\right)e^{-ixm/\sqrt{n}}dx-\frac{1}{2\pi}\int_{\mathbb{R}} e^{-\sigma^2x^2/2}e^{-ixm/\sqrt{n}}dx\right|=o(1).
\]
Using the triangle inequality and
\[
\int_{|x|\geq \pi\sqrt{n}}e^{-\sigma^2x^2/2}dx\xrightarrow[n\to\infty]{}0,
\]
it suffices to show that
\begin{equation}\label{eq: LLT suff}
\int_{-\pi\sqrt{n}}^{\pi\sqrt{n}}\left|\phi_n\left(\frac{x}{\sqrt{n}}\right)-e^{-\sigma^2x^2/2}\right|dx\xrightarrow[n\to\infty]{}0.
\end{equation}
Since $\frac{\mathsf{U}_n}{\sqrt{n}}$ converges in distribution to a centered normal with variance $\sigma^2$, it follows from Levy's continuity theorem that for all $x\in \mathbb{R}$,
\[
\Psi_n(x):=\mathbf{1}_{[-\pi\sqrt{n},\pi\sqrt{n}}](x)\left|\phi_n\left(\frac{x}{\sqrt{n}}\right)-e^{-\sigma^2x^2/2}\right|\xrightarrow[n\to\infty]{}0.	
\]
In addition, by Lemmas \ref{lem: dom away from 0} and \ref{lem: dom near 0}, for every $n$ large, for all $|x|\leq\pi\sqrt{n}$, 
\[
\Psi_n(x)\leq G(x)
\]
where
\[
G(x)=e^{-\sigma^2x^2/2}+\max\left(\exp\left(-Lx^2\right),\exp\left(-d\sqrt{|x|}\right)\right).
\] 
and $L$ and $d$ are the constants in the Lemmas. Since $G$ is integrable, it follows from the dominated convergence theorem that 
\[
\int_{-\pi\sqrt{n}}^{\pi\sqrt{n}}\left|\phi_n\left(\frac{x}{\sqrt{n}}\right)-e^{-\sigma^2x^2/2}\right|dx=\int_{\mathbb{R}}\Psi_n(x)dx\xrightarrow[n\to\infty]{}0.
\]
The proof is thus concluded. 
\end{proof}
\begin{remark}
For $k\in I_n$ and $j\leq n$ we let $X_{k}(j)$ be an i.i.d. sequence which is distributed as
\[
p_kU^j\bar{f}_k+p_{k+1}U^j\bar{f}_{k+1}.
\]
The random variable $X_k(j)$ takes values in $\left\{0,\pm 1,\pm 2^k,\pm \left(2^k+1\right),\pm  \left(2^{k+1}+1\right)\right\}$.\\
We assume $\left\{X_k(j)\right\}_{k\in I_n,\ j\leq n}$ are independent. Note that writing
\[
\mathsf{X}_n:=\sum_{k\in I_n}\sum_{j=2^k}^nX_k(j),
\]
then if $\mathsf{X}_n,\mathsf{X}_n'$ are independent identically distributed, then $\mathsf{U}_n \overset{d}{=} \mathsf{X}_n+\mathsf{X}_n'$. Therefore 
\[
\phi_n(t)=\prod_{k\in I_n}\left[\mathbb{E}\left(\exp(itX_k(1))\right)\right]^{2\left(n-2^k\right)}
\]
\end{remark}
\begin{proof}[Proof of Lemma \ref{lem: dom away from 0}]
	First note that for all $k\in I_n$
	\[
	\mathbb{P}\left(X_{k}(1)=z\right)=\begin{cases}
	\left(1-\alpha_{k}^{2}\right)\left(1-\alpha_{k+1}^{2}\right), & z=0\\
	\frac{\alpha_{k}^{2}}{2}\left(1-\alpha_{k+1}^{2}\right) & z\in\left\{ \pm2^{k}\right\} \\
	\frac{\alpha_{k+1}^{2}}{2}\left(1-\alpha_{k}^{2}\right) & z\in\left\{ \pm\left(2^{k}+1\right)\right\} \\
	\frac{\left(\alpha_{k}\alpha_{k+1}\right)^{2}}{4} & z\in\left\{ \pm1,\pm\left(2^{k+1}+1\right)\right\} 
	\end{cases}
	\]
	so substituting $t=\frac{x}{\sqrt{n}}\in[-\pi,\pi]$, 
	\begin{align*}
	\left|\phi_{n}\left(\frac{x}{\sqrt{n}}\right)\right| &=\prod_{k\in I_{n}}\left|\mathbb{E}\left(\exp(ix X_k(1)/\sqrt{n}\right)^{2\left(n-2^{k}\right)}\right|\\
	&=\prod_{k\in I_{n}}\left|\mathbb{E}\left(\exp(it X_k(1)\right)\right|^{2\left(n-2^{k}\right)}\\
	& \leq \prod_{k\in I_{n}}\left(1-\frac{\left(\alpha_{k}\alpha_{k+1}\right)^{2}}{2}\left(1-\cos t\right)\right)^{2\left(n-2^{k}\right)}.
	\end{align*}
	For the bound of $\sqrt[4]{n}\leq|x|\leq\pi\sqrt{n}$ note that for
	such $x$, 
	\[
	\cos\left(t\right)\leq\cos\left(\frac{1}{\sqrt[4]{n}}\right)\leq1-\frac{1}{4\sqrt{n}}.
	\]
	This shows that for  $\sqrt[4]{n}\leq|x|\leq\pi\sqrt{n}$, 
	\begin{align*}
	\left|\phi_{n}\left(\frac{x}{\sqrt{n}}\right)\right|& \leq\prod_{k\in I_{n}}\left(1-\frac{\left(\alpha_{k}\alpha_{k+1}\right)^{2}}{8\sqrt{n}}\right)^{2(n-2^{k})}\\
	&\leq\exp\left(-\frac{1}{4\sqrt{n}}\left(\sum_{k\in I_{n}}n\left(\alpha_{k}\alpha_{k+1}\right)^{2}-\sum_{k\in I_{n}}\left(2^{k}\right)\alpha_{k}^{4}\right)\right).
	\end{align*}
	A calculation shows that 
	\begin{align*}
	\sum_{k\in I_{n}}\left(2^{k}\right)\alpha_{k}^{4}\leq \sum_{k\in I_{n}}\left(2^{k}\right)2^{-4k}=o(1),\ \text{as}\ n\to\infty,
	\end{align*}
	and
	\begin{align*}
	\sum_{k\in I_{n}}\left(\alpha_{k}\alpha_{k+1}\right)^{2} & \sim\sum_{k\in2\mathbb{N}:\ \sqrt{\log n}\leq k\leq\log n}\alpha_{k}^{4}\\
	& \geq\sum_{k\in\mathbb{N}:\ \sqrt{\log n}\leq2k\leq\log n}32^{-2k}\\
	&\geq C2^{-10\left(\sqrt{\log n}/2\right)}
	=C2^{-5\sqrt{\log n}}
	\end{align*}
	for some global constant $C$ which does not depend on $n$. Since
	\[
	\sqrt{n}2^{-5\sqrt{\log n}}\gtrsim\sqrt[4]{n}
	\]
	we have shown that there exists $c>0$ such that for all $\sqrt[4]{n}\leq x\leq\pi\sqrt{n}$, 
	
	\[
	\left|\phi_{n}(x)\right|\leq\exp\left(-c\sqrt[4]{n}\right)\leq \exp\left(-d\sqrt{|x|}\right).
	\]
	where $d=\frac{c}{\sqrt{\pi}}$.
\end{proof}
Define for $n\in\mathbb{N},$ 
\[
J_{n}:=\left\{ k\in\mathbb{N}:2^{k}\leq\frac{\sqrt[4]{n}}{3},\ n<2^{k^{2}}\right\} =\left(\sqrt{\log n},\frac{\log n}{4}-\log3\right]\cap\mathbb{N}.
\]
and for $1\leq j\leq n$, 
\[
\Upsilon_{n}(j)=\sum_{k\in J_{n}:\ 2^k\leq j}p_{k}U^{j}\bar{f}_{k}
\]
and
finally $\mathsf{Z}_{n}=\sum_{j=1}^{n}\Upsilon_n(j)$. Note that $\left\{ \Upsilon_{n}(j)\right\} _{j=1}^{n}$
are independent and that $\mathsf{U}_n-\mathsf{Z}_{n}$ and $\mathsf{Z}_n$
are independent by the construction. 
\begin{lemma}
	\label{lem: small t}There exists $N\in\mathbb{N}$ and a constant
	$L>0$ such that for all $n>N$ and $|x|\leq\sqrt[4]{n}$,
	\[
	\left|\mathbb{E}\left(\exp(ix\mathsf{Z}_n/\sqrt{n}\right)\right|\leq\exp\left(-Lx^{2}\right)
	\]
\end{lemma}

\begin{proof}
	Now $\Upsilon_{n}(j),\ 1\leq j\leq n$ are independent, for all $1\leq j\leq n$,
	$\mathbb{E}\left(\Upsilon_{n}(j)\right)=0$ and 
	\begin{align*}
	\mathbb{E}\left(\left(\Upsilon_{n}(j)\right)^{2}\right) & =\sum_{\left\{ k\in J_{n}:2^k\leq j\right\} }p_{k}^{2}\alpha_{k}^{2}\\
	& =\sum_{\left\{ k\in J_{n}:2^k\leq j\right\} }\frac{1}{k\log k}.
	\end{align*}
	For all $\frac{n}{2}\leq j\leq n$,\footnote{For all $k\in J_{n},\ 2^{k}\leq\sqrt[4]{n}/3<n/2$}
	\begin{equation}
	\sum_{\left\{ k\in J_{n}:2^k\leq j\right\} }\frac{1}{k\log k}=\sum_{k\in J_{n}}\frac{1}{k\log k}.\label{eq:Heart}
	\end{equation}
	A similar argument as in the proof of Lemma \ref{lem: 2nd for CLT} shows that there exists $\beta>0$
	such that 
	\begin{align*}
	\sum_{k\in J_{n}}\frac{1}{k\log k} & =\sum_{\sqrt{\log n}< k\leq\frac{\log n}{4}-\log3}\frac{1}{k\log k}\\
	&\sim \int_{\sqrt{\log(n)}}^{\frac{\log(n)}{4}-\log3}\frac{dx}{x\log x}\\
	& =\ln2\left(\ln\ln\left(\frac{\log n}{4}-\log3\right)-\ln\ln\left(\sqrt{\log n}\right)\right)\\
	& =\ln2\left(\ln\ln(\log n)-\ln\left(\frac{\ln(\log(n))}{2}\right)\right)+o(1)=(\ln2)^2+o(1).
	\end{align*}
	Consequently, for all $\frac{n}{2}\leq j\leq n$, 
	\begin{equation}\label{eq: 2nd mom in small}
	\mathbb{E}\left(\left(\Upsilon_{n}(j)\right)^{2}\right)\sim \left(\ln2\right)^2,\ \ \text{as}\ n\to\infty,
	\end{equation}
	Note that the latter asymptotic equivalence is uniform when $n\to\infty$ and $\frac{n}{2}\leq j\leq n$. \\
	Since for all $k\in J_{n}$, $$p_{k}\leq 2^k\leq \frac{\sqrt[4]{n}}{3},$$ for all $k\in J_{n}$
	and $1\leq j\leq n$,
	\[
	\mathbb{E}\left(\left(p_{k}U^{j}\bar{f}_{k}\right)^{4}\right)=p_{k}^{4}\alpha_{k}^{2}=\frac{p_{k}^{2}}{k\log k}\leq\frac{\sqrt{n}}{9}\frac{1}{k\log k}.
	\]
	By this, the independence of $\left\{ p_{k}U^{j}\bar{f}_{k}\right\} _{k\in J_{n}}$
	and that for all $k\in J_{n}$ and $j\in\mathbb{N}$, $\mathbb{E}\left(p_{k}U^{j}\bar{f}_{k}\right)=0$
	we see that for all $\frac{n}{2}\leq j\leq n$, 
	\begin{align*}
	\mathbb{E}\left(\left(\Upsilon_{n}(j)\right)^{4}\right) & =\mathbb{E}\left(\left(\sum_{k\in J_{n}}p_{k}U^{j}\bar{f}_{k}\right)^{4}\right)\\
	& \leq\sum_{k\in J_{n}}\mathbb{E}\left(\left(p_{k}U^{j}\bar{f}_{k}\right)^{4}\right)+\left(\sum_{k\in J_{n}}\mathbb{E}\left(\left(p_{k}U^{j}\bar{f}_{k}\right)^{2}\right)\right)^{2}\\
	& \leq\frac{\sqrt{n}}{9}\sum_{k\in J_{n}}\frac{1}{k\log k}+\left(\mathbb{E}\left(\left(\Upsilon_{n}(j)\right)^{2}\right)\right)^{2}\\
	&=\frac{\sqrt{n}}{9}\left(\ln 2\right)^2+o\left(\sqrt{n}\right).
	\end{align*}
There exists $N$ such that for all $n>N$, and $\frac{n}{2}\leq j\leq n$
	\begin{equation}
	\mathbb{E}\left(\left(\Upsilon_{n}(j)\right)^{2}\right)\geq\frac{(\ln2)^2}{2}\label{eq:jjj}
	\end{equation}
	and in addition
	\[
	\mathbb{E}\left(\left(\Upsilon_{n}(j)\right)^{4}\right)\leq\frac{2(\ln2)^2}{9}\sqrt{n}\leq\frac{4\sqrt{n}}{9}\mathbb{E}\left(\left(\Upsilon_{n}(j)\right)^{2}\right).
	\]
	Furthermore, $\left\{U^i\bar{f}_k\right\}_{\left\{k\in J_n:\ 2^k\leq j\right\}}$ is a sequence of independent and symmetric random variables, thus
	\[
	\mathbb{E}\left(\Upsilon_n(j)\right)=\mathbb{E}\left(\left(\Upsilon_n(j)\right)^3\right)=0.
	\] 
	It follows that for all $n/2\leq j\leq n$, \footnote{See for example \cite{DurrPT2}[pp. 101-103]} and $|t|\leq\frac{1}{\sqrt[4]{n}}$
	\begin{align*}
	\left|\mathbb{E}\left(\exp\left(it\Upsilon_{n}(j)\right)\right)\right| & \leq 1-\frac{E\left(\left(\Upsilon_{n}(j)\right)^{2}\right)}{2}t^{2}+\frac{t^{4}}{4!}\mathbb{E}\left(\left(\Upsilon_{n}(j)\right)^{4}\right)\\
	& \leq1-E\left(\left(\Upsilon_{n}(j)\right)^{2}\right)t^{2}\left(\frac{1}{2}-\frac{4\sqrt{n}}{9}t^{2}\right)\ \ \ \sqrt{n}t^{2}\leq1\ \text{and}\ \eqref{eq:jjj}\\
	& \leq1-\frac{(\ln2)^2}{36}t^{2}\leq\exp\left(-\frac{(\ln 2)^2}{36}t^{2}\right).
	\end{align*}
	
	Finally by independence of $\left\{\Upsilon_{n}(j)\right\} _{j=1}^{n}$,
	for all $|x|\leq\sqrt[4]{n}$, 
	\begin{align}
	\left|\mathbb{E}\left(\exp(ix\mathsf{Z}_n/\sqrt{n}\right)\right| & =\prod_{j=1}^{n}\left|\mathbb{E}\left(\exp(ix\Upsilon_{n}(j)/\sqrt{n}\right)\right|\label{eq: CF}\\
	& \leq\prod_{j=n/2}^{n}\left|\mathbb{E}\left(\exp(ix\Upsilon_{n}(j)/\sqrt{n}\right)\right|\nonumber \\
	& \leq\prod_{j=n/2}^{n}\exp\left(-\frac{(\ln2)^2}{36n}x^{2}\right)=\exp\left(-\frac{(\ln 2)^2}{72}x^{2}\right).\nonumber 
	\end{align}
	The conclusion follows with $L=\frac{(\ln2)^2}{72}$. 
\end{proof}
\begin{proof}[Proof of Lemma \ref{lem: dom near 0}]
Let $N$ and $L$ be as in Lemma \ref{lem: small t}. Since $\mathsf{Z}_n$  and $\mathsf{U}_n-\mathsf{Z}_n$ are independent, then for all $n>N$, and $|x|\leq \sqrt[4]{n}$,
\begin{align*}
	\left|\mathbb{E}\left(\exp(ix\mathsf{U}_n/\sqrt{n}\right)\right|&=\left|\mathbb{E}\left(\exp(ix\left(\mathsf{U}_n-\mathsf{Z}_n\right)/\sqrt{n}\right)\right| \left|\mathbb{E}\left(\exp(ix\mathsf{Z}_n/\sqrt{n}\right)\right| \\
	&\leq 	\left|\mathbb{E}\left(\exp(ix\mathsf{Z}_n/\sqrt{n}\right)\right|\\
	& \leq \exp\left(-Lx^2)\right.
\end{align*}
\end{proof}
\section{Appendix} 
The first result in the appendix is that the local limit theorem persists under addition of small independent noise. These type of arguments and statements are not new. We include a statement which is especially tailored for our construction. 
\begin{proposition}\label{prop: essential part}
Suppose that $X_n=Y_n+Z_n$ which are for each $n\in\mathbb{N}$, $Y_n$ and $Z_n$ are $\mathbb{Z}$-valued independent random variables. If 
\begin{equation}\label{eq: LLT APP}
\sup_{x\in\mathbb{Z}}\left|\sqrt{n}\mathbb{P}\left(Y_n=x\right)-\frac{1}{\sqrt{2\pi\sigma^2}}e^{-x^2/2n\sigma^2}\right|\xrightarrow[n\to\infty]{}0
\end{equation}
and $\mathbb{E}\left(|Z_n|^2\right)=O\left(\frac{n}{\sqrt{\log(n)}}\right)$ then 
 \[
 \sup_{x\in\mathbb{Z}}\left|\sqrt{n}\mathbb{P}\left(X_n=x\right)-\frac{1}{\sqrt{2\pi\sigma^2}}e^{-x^2/2n\sigma^2}\right|\xrightarrow[n\to\infty]{}0
 \]
\end{proposition}
In order to simplify the notation in the proof, we would make use of the following reformulation of \eqref{eq: LLT APP}: There exists $r:\mathbb{N}\to [0,\infty)$ such that $r(n)\sqrt{n}\to 0$ and for all $x\in \mathbb{Z}$, 
\[
\mathbb{P}\left(Y_n=x\right)=\frac{1}{\sqrt{2\pi n\sigma^2}}e^{-\frac{x^2}{2n\sigma^2}}\pm r(n). 
\]
In the course of the proof, the function $r(n)$ will denote a $o\left(\frac{1}{\sqrt{n}}\right)$ function which may change from line to line. 
\begin{proof}
By changing from $X_n$ to $\frac{1}{\sigma}X_n$ we can and will assume that $\sigma^2=1$. Write $a(n)=\left(\frac{n}{\sqrt[4]{\log(n)}}\right)^{\frac{1}{2}}$. By Markov inequality,
\[
\mathbb{P}\left(\left|Z_n\right|\geq a(n)\right)\leq \frac{D}{\sqrt[4]{\log(n)}},
\]
where $D$ is any constant such that $\sup_{n\in\mathbb{N}}\left[\frac{\sqrt[4]{\log(n)}}{n}\mathbb{E}\left(\left|Z_n\right|^2\right)\right]\leq D$.\\
Fix $x\in\mathbb{Z}$ and note that as $Y_n$ and $Z_n$ are independent,
\[
\mathbb{P}\left(X_n=x\right)=\sum_{z\in\mathbb{Z}}\mathbb{P}\left(Y_n=x-z\right)\mathbb{P}\left(Z_n=z\right).
\]
We split the sum into $|z|\leq a_n$ and $|z|>a_n$. \\
A consequence of \eqref{eq: LLT APP} is that there exists $C>0$ such that for all $y\in\mathbb{Z}$ and $n\in\mathbb{N}$,
\[
\mathbb{P}\left(Y_n=y\right)\leq \frac{C}{\sqrt{n}}. 
\]
And,
\begin{equation}\label{eq: first bound in LLT APP}
\sum_{z\in\mathbb{Z}\setminus \left[-a(n),a(n)\right]}\mathbb{P}\left(Y_n=x-z\right)\mathbb{P}\left(Z_n=z\right)\leq \frac{C}{\sqrt{n}}\mathbb{P}\left(\left|Z_n\right|\geq a(n)\right)\leq \frac{CD}{\sqrt{n}\sqrt[4]{\log(n)}}.
\end{equation}
If $|z|\leq a(n)$ and $|x|>\sqrt{n}\log^{1/9}(n)$ then for all large $n$, uniformly in $z=\pm a(n)$, 
\[
\exp\left(-\frac{(x-z)^2}{2n}\right)\leq \exp\left(-\frac{x^2}{4n}\right)=\exp\left(-\frac{x^2}{2n}\right)+o(1).
\]
The last equality up to a $o(1)$ term follows from $x^2/n\geq  \log^{2/9}(n)\to \infty$ as $n\to\infty$.\\
If $|x|\leq \sqrt{n}\log^{1/9}(n)$ then $|xz|\leq\frac{n}{(\log(n))^{\frac{1}{72}}}$, thus
\begin{align*}
\exp\left(-\frac{(x-z)^2}{2n}\right) &= \exp\left(-\frac{x^2}{2n}\right)\exp\left(\pm \frac{|xz|)}{2n}\right)\\
&=\exp\left(-\frac{x^2}{2n}\right)\exp\left(\pm \frac{1}{2(\log(n))^{\frac{1}{72}}}\right)\\
&=\exp\left(-\frac{x^2}{2n}\right)+o(1).
\end{align*}
Using these rather trivial bounds and the reformulation of \eqref{eq: LLT APP} we conclude that 
\begin{align*}
\sum_{z\in\mathbb{Z}\cap \left[-a(n),a(n)\right]}\mathbb{P}\left(Y_n=x-z\right)\mathbb{P}\left(Z_n=z\right)&= \sum_{z\in\mathbb{Z}\cap \left[-a(n),a(n)\right]} \frac{\exp\left(-\frac{(x-z)^2}{2n}\right))}{\sqrt{2\pi n}}\mathbb{P}\left(Z_n=z\right)\pm r(n)\\
&=\sum_{z\in\mathbb{Z}\cap \left[-a(n),a(n)\right]} \left[\frac{\exp\left(-\frac{x^2}{2n}\right)}{\sqrt{2\pi n}}+o(1)\right]\mathbb{P}\left(Z_n=z\right)\pm r(n)\\
&=\mathbb{P}\left(\left|Z_n\right|\leq a_n\right)\left[\frac{1}{\sqrt{2\pi n}}\exp\left(-\frac{x^2}{2n}\right)\right] \pm (o(1)+r(n))\\
\end{align*}
As $\mathbb{P}\left(\left|Z_n\right|\leq a_n\right)=1+o(1)$,
\[
\sum_{z\in\mathbb{Z}\cap \left[-a(n),a(n)\right]}\mathbb{P}\left(Y_n=x-z\right)\mathbb{P}\left(Z_n=z\right)=\frac{1}{\sqrt{2\pi n}}\exp\left(-\frac{x^2}{2n}\right)+o(1).
\] 
The conclusion follows from the latter asymptotic equality and \eqref{eq: first bound in LLT APP}. 
\end{proof}
The following estimate is used in bounding the $L^2$ norm of the first term in Proposition \ref{prop: decomp for CLT}.
\begin{lemma}\label{lem: app1}
There exists a constanst $K>0$ such that for all $n\in\mathbb{N}$,
$$\sum_{k\in\mathbb{N}:\ 2\leq k\leq \sqrt{\log n}}\frac{2^{k^2}}{k\log k}\leq K\frac{n}{\sqrt{\log(n)}}. $$
\end{lemma}
\begin{lemma}\label{lem: app2}
	$$\sum_{k: 2^k\leq n}\frac{2^k}{k}=O\left(\frac{n}{\log n}\right).$$
\end{lemma}
\begin{proof}[Proof of Lemmas \ref{lem: app2} and \ref{lem: app2}]
Both results follow from the following reasoning. If $\left(a_k\right)_{k=1}^\infty$ is a sequence of positive reals such that there exists $q>1$ for which for all $n\in\mathbb{N}$, $a_{n+1}/a_n\geq q$ then for all $n\in\mathbb{N}$,
\[
\sum_{k=1}^na_k\leq a_n\sum_{k=1}^\infty q^{-k}\leq \frac{a_n}{1-q}.
\]
\end{proof}

\bibliographystyle{abbrv}
\bibliography{embedding}
\end{document}